\documentclass[reqno]{amsart}
\usepackage[T1]{fontenc}
\usepackage{dsfont}
\usepackage{mathrsfs}
\usepackage[colorlinks]{hyperref}
\usepackage{xcolor}
\usepackage[a4paper,asymmetric]{geometry}
\usepackage{mathscinet}
\usepackage{latexsym}
\usepackage{amsthm}
\usepackage{amssymb}
\usepackage{amsfonts}
\usepackage{amsmath}
\usepackage{longtable}
\usepackage{graphicx}
\usepackage{multirow}
\usepackage{multicol}

\newtheorem{theorem}{Theorem}[section]
\newtheorem{thm}[theorem]{Theorem}
\newtheorem{lem}[theorem]{Lemma}

\newtheorem{prop}[theorem]{Proposition}
\newtheorem{hypo}[theorem]{Hypothesis}
\newtheorem{ex}[theorem]{Example}
\newtheorem{rem}[theorem]{Remark}
\newtheorem{de}[theorem]{Definition}
\theoremstyle{remark}
\numberwithin{equation}{section}

\newcommand{\R}{\mathbb{R}}
\newcommand{\N}{\mathbb{N}}
\newcommand{\Z}{\mathbb{Z}}
\newcommand{\E}{\mathbb{E}}

\newcommand{\abs}[1]{\left\vert#1\right\vert}
\newcommand{\set}[1]{\left\{#1\right\}}
\newcommand{\norm}[1]{\left\Vert#1\right\Vert}
\allowdisplaybreaks

\begin{document}
\title[ Second Moment Estimator for AR(1) Model]{Second moment estimator for an AR(1) model driven by a long memory Gaussian noise
}
\author[Y. Chen]{Yong CHEN}
 \address{School of Mathematics and Statistics, Jiangxi Normal University, Nanchang, 330022,  Jiangxi,  P. R. China}
\email{zhishi@pku.org.cn}
 \author[L. Tian]{Li TIAN}
 \address{School of Mathematics and Statistics, Jiangxi Normal University, Nanchang, 330022, Jiangxi, P. R. China}
\email{tianli@jxnu.edu.cn}
  \author[Y. LI]{Ying LI}   \address{School of Mathematics and Computational Science, Xiangtan University, Xiangtan, 411105, Hunan, P. R. China (Corresponding author.)}
  \email{liying@xtu.edu.cn}

\begin{abstract}
In this paper, we consider an inference problem for the first order autoregressive process driven by a long memory stationary Gaussian process. Suppose that the covariance function of the noise can be expressed as $\abs{k}^{2H-2}$ times a function slowly varying at infinity.
The fractional Gaussian noise and the fractional ARIMA model and some others Gaussian noise are special examples that satisfy this assumption.
We propose a second moment estimator and prove the strong consistency, the asymptotic normality and the almost sure central limit theorem. Moreover,  we give the upper Berry-Ess\'een bound by means of Fourth moment theorem.
\end{abstract}

\maketitle


\section{ Introduction}
For the first order autoregressive model $ (X_t, t\in \N) $ driven by a given noise sequence $\xi= (\xi_t, t\in \Z)$:
\begin{align}\label{model}
     X_{t} = \theta X_{t-1}  + \xi_{t}, \quad t\in \N
\end{align}
with  $X_0=0$, the inference problem regarding
the parameter $\theta$  has been extensively studied in probability and statistics literatures.
For when $\xi$ is independent identical distribution or a martingale difference sequence, this problem has been widely studied over the past decades
(see \cite{Anderson 79,Lai 83} and the references therein).
In the heavy-tailed noise case, the least square estimator (LSE) of AR(p) models was studied in \cite{Zhang}.
The maximum likelihood estimator (MLE) of AR(p) was investigated in \cite{Brouste} for regular stationary Gaussian noise, in which they transform the observation model into an ``equivalent''  model with Gaussian white noise. In \cite{Brouste},  it is pointed out that for the strongly dependent noises the LSE is generally not consistent.  Very recently,  in case of long-memory noise, the detection of a change of the above parameter $\theta$ is studied by means of the likelihood ratio test  \cite{Brouste 20}.

In this paper, we will discuss the long-range dependence Gaussian noise case and propose a second moment estimator.  First, we find that it is very convenient to construct a second moment estimator when we restrict the domain of the parameter $\theta$  in $$\Theta = \{  \theta \in \R\mid 0<\theta<1   \}.$$   It  seems that this restriction is very reasonable for real-world context sometimes. In fact, $\abs{\theta}<1$ is an assumption to ensure the model (\ref{model}) to have a stationary solution. We rule out the case of $-1<\theta<0$ in which  the series tends to oscillate rapidly. We also rule out the case of $\theta= 0$ in which $X_t$ is not an autoregressive model  any more.

Next,  we assume that the  stationary Gaussian nosie $\xi$ satisfies the following Hypothesis \ref{assume}:

\begin{hypo}\label{assume}
        The covariance function $ \rho(k)= \E(\xi_0 \xi_k) $ for any $k\in \Z$ satisfies
\begin{equation}\label{cov}
   \rho(k) = L(k)  \vert k \vert ^{2H-2} ,   \quad  H \in \left (\frac 12 , 1  \right)
\end{equation}   with $ L: (0, \infty)\to (0, \infty)$ is slowly varying at infinity in Zygmund's sense and $\tilde{L}(\lambda):=L(\frac{1}{\lambda}) $ is of bounded variation on $(a,\,\pi)$ for any $a> 0$.
Moreover, $ \rho (0)=1 $.
 \end{hypo}
It is well known that Eq. (\ref{cov}) is equivalent to the spectral density of $\xi$ satisfying
    \begin{equation*}
 h_{\xi}(\lambda) \sim C_H L(\lambda ^{-1}) \vert \lambda \vert ^{1-2H} , \quad \text{as} \; \lambda \to 0
    \end{equation*} with $C_H = \pi^{-1} \Gamma(2H-1) \sin ( \pi -\pi H )$. Please refer to \cite{Beran} or Lemma \ref{sp-cov} below.

We will see that the fractional Gaussian noise, the fractional ARIMA model driven by Gaussian white noise
and some other long memory Gaussian processes are special examples satisfying Hypothesis~\ref{assume}.

When $\vert  \theta  \vert < 1 $, the stationary solution to the model (\ref{model}) is $$Y_t=\sum_{j=0}^{\infty} \theta^j  \xi_{t-j},$$
and the solution with initial value $X_0=0$ can be represented as:
     \begin{equation}\label{solution}
 X_t = Y_t + \theta^t\zeta  ,
     \end{equation} where $\zeta$ is a normal random variable with zero mean.
It is clear that the second moment of $Y_t$ is:
\begin{align*}
      f(\theta) :=  \E (Y_t^2 )=  \sum_{i,j=0}^\infty   \theta^{i+j} \rho(i-j) .
\end{align*}
If $0<\theta<1$ then $f(\theta)$ is positive and strictly increasing.  Denote $f^{-1}(\cdot)$ is the inverse function of $f(\cdot)$.
We propose a second moment estimator of $\theta$ as:
           \begin{equation}\label{estimator}
    \tilde{\theta} _n   = f^{-1}\left( \frac{1}{n}\sum_{t=1}^{n}X_t^2\right).
           \end{equation}

  In this paper, we will show the strong consistency and give the asymptotic distribution for that estimator.
Moreover, we also give the Berry-Ess\'een bound when the limit distribution is Gaussian. These results are stated in the following theorems:
\begin{thm}\label{consist}
Under Hypothesis \ref{assume},
the estimator $\tilde{\theta}_n$ is strongly consistent, i.e.,
\begin{equation*}
  \lim_{n\to \infty} \tilde{\theta}_n   =\theta \quad \text{a.s.} .
\end{equation*}
\end{thm}
\begin{thm}\label{normal}
Under Hypothesis \ref{assume} and suppose $H\in (\frac 12, \frac 34)$, we have the following asymptotic distribution of $\tilde{\theta}_n$ as $n\to \infty$:
\begin{equation}\label{1}
\sqrt{n}(\tilde{\theta}_n  -\theta)  \xrightarrow{law}
\mathcal{N} \left( 0 , \frac{\sigma_H^2}{[f'(\theta)]^2 } \right),
\end{equation} where $\sigma_H^2=2\sum_{k\in \Z} R^2(k)$ and $f'(\theta)$ is the derivative of $f(\theta)$.
\end{thm}
\begin{rem}
The case of $H\in [\frac34, 1)$ will be treated in a separate paper.
\end{rem}

\begin{thm}\label{ASCLT 01}
Let $Z$ be a standard Gaussian random variable.  Under Hypothesis \ref{assume} and suppose $H\in (\frac 12, \frac 34)$, then
$$G_n:= \frac{ f'(\theta) \sqrt n  (\tilde{\theta}_n  -\theta)}{ \sigma_H }$$
satisfies an almost sure central limit theorem (ASCLT). In other words, almost surely, for all $z\in \R$,
\begin{equation*}
\frac{1}{\log n}\sum_{k=1}^n \frac{1}{k}  \mathbf{1}_{\set{G_k\le z}}\to P(Z\le z) \quad \mathrm{ as } \quad n\to \infty,
\end{equation*}
or, equivalently, almost surely, for all continuous and bounded functions $\varphi:\,\R\to \R$,
\begin{equation*}
\frac{1}{\log n}\sum_{k=1}^n \frac{1}{k} \varphi(G_k)\to \E[\varphi (Z)]\quad \mathrm{ as } \quad n\to \infty.
\end{equation*}
\end{thm}
\begin{thm}\label{B-E}
Let $Z$ be a standard Gaussian random variable. 
Under Hypothesis \ref{assume} and suppose $H\in (\frac 12, \frac 34)$, there exists a constant $C_{\theta, H}> 0$ such that when $n$ is large enough,
\begin{equation}\label{l34}
  \sup_{z\in \R} \left|  P\left\{ \frac{ f'(\theta) \sqrt n  (\tilde{\theta}_n  -\theta)}{ \sigma_H } \le z \right\}
-   P\{ Z\le z \}  \right|  \le C_{ \theta,H}  \varphi (n),
\end{equation} 
where $\sigma_H$ is given in Theorem~\ref{normal} and
\begin{equation*}
  \varphi (n) =\left\{
         \begin{array}{ll}
\frac{1}{n ^ { \frac12- }},   & \quad \textit{if} \,\,   H\in   \left( \frac 12, \frac 58 \right) ;\\
\frac{1}{n ^ {3-4H-}} ,    & \quad \textit{if} \,\,  H\in \left[ \frac 58, \frac 34 \right).
            \end{array}
   \right.
\end{equation*} Here $ \frac12-$ means that $ \frac12-\epsilon$ for any $\epsilon>0$.
\end{thm}
\begin{rem}
We do not know how to improve the upper bound $1/n ^ { \frac12- }$ to $1/\sqrt{n}  $ when $H\in   \left( \frac 12, \frac 58 \right)$. 
\end{rem}
Next, we give some sequences that satisfy Hypothesis \ref{assume}.

\begin{ex}
The fractional Gaussian noise with covariance function
\begin{equation*}
  \rho(k)=\frac{1}{2}\left(\vert k+1\vert^{2H}+\vert k-1\vert^{2H}-  2\vert k \vert^{2H}   \right), \quad k \in \Z
\end{equation*}
satisfies Hypothesis \ref{assume} when $ H > \frac 12 $. It is well-known that $ \rho(k) \sim  H(2H-1) | k |^{2H-2} $ as $k\to \infty$.
\end{ex}

\begin{ex}
The fractional ARIMA(0,$d$,0) model driven by a Gaussian white noisesatisfies Hypothesis \ref{assume} when $d\in (0,\,\frac12)$. It is well-known that  it's the spectral density satsifies
\begin{equation*}
  h_{\xi}(\lambda) \sim \frac 1{2\pi}\vert \lambda \vert^{-2d},  \quad \lambda \to 0
\end{equation*}

\end{ex}

\begin{ex}
$\xi_t$ is linear stationary sequence defined by
\begin{equation*}
    \xi_t  =  \sum_{k\le t}  b_{t-k} \varepsilon_t , \quad t\in \Z .
\end{equation*}
Here $\set{\varepsilon _t, t\in \Z}$ is a Gaussian white noise, and the  weights $\set{b_t, t\in \N_+}$ satisfy that $\sum b_t^2  < \infty$. Assume that the weights decay slowly hyperbolically:
\begin{equation*}
    b_k  =  L_0(k) \vert k \vert ^{H- \frac 32} ,
\end{equation*}
where $H\in (\frac 12 , 1)$, and $L_0( \cdot )$ is a slowly varying function. Then Hypothesis \ref{assume} is valid if we take $L(k) \propto L_0^2(k)$ \textup{(see \cite{Giraitis 96})}.
\end{ex}

     In the remainder of this paper, $C$ and $c$ will be a generic positive constant independent of $n$ the value of which may differ from line to line.
\section{ Preliminary}

In this section, we list the main definitions and theorems that is used to show our results.
The following two definitions are cited from Definition~1.1 and 1.2 of \cite{Beran} respectively.
\begin{de}
A positive function $L(x)$ defined for $x>x_0$ is called a slowly varying at infinity function
in Zygmund's sense if, for any $\delta > 0$,
$p_1(x) = x^{\delta}L(x)$ is an increasing, and $p_2(x) = x^{ - \delta}L(x)$ is a decreasing, function of $x$ for $x$ large enough.
Similarly, $L$ is called slowly varying at the origin if $\tilde{L}(x)  = L(x^{-1})$ is slowly varying at infinity.
\end{de}
It is known that if $L(x) $ is slowly varying at infinity then
\begin{align}\label{karamata}
\lim_{x\to\infty} \frac{L(ux)}{L(x)}=1
\end{align}
for every fixed $u>0$, and even uniformly in every interval $a\le u\le \frac{1}{a},\,0<a<1$ \cite[p.186]{Zygmund}.

\begin{de}\label{delong}
Let $\set{\xi_t}$ be a second-order stationary process with autocovariance function $\rho(k)(k\in \Z)$ and spectral density
\begin{equation*}
  h_{\xi} (\lambda) =  (2\pi)^{-1}  \sum_{k=-\infty}^{\infty} \rho(k) \exp(-ik\lambda) ,  \quad \lambda \in [-\pi, \pi].
\end{equation*}
Then $\set{\xi_t}$ is said to exhibit linear long-range dependence, if
\begin{equation*}
  h_{\xi}  (\lambda)  = L_h(\lambda)  \vert \lambda \vert ^{1-2H} ,
\end{equation*}
where $L_h(\lambda) > 0$ is a symmetric function that is slowly varying at zero and $H\in (\frac 12, 1)$.
\end{de}

The following theorem is well-known and is cited from Theorem 1.3 of \cite{Beran}:
\begin{thm}\label{sp-cov}
Let $ R(k)(k\in \Z )$ and $h(\lambda) (\lambda \in [\pi, \pi] )$ be the autocovariance function
and spectral density respectively of a second-order stationary process $\set{\xi_t}$. Then the following holds:
                      \begin{enumerate}
            \item If  \begin{equation*}
  R(k)=L_R(k) \vert k \vert ^{2H-2}, \quad k\in \Z,
\end{equation*}
where $L_R(k)$ is slowly varying at infinity in Zygmund's sense, and
 $H\in(\frac 12 , 1)$, then
\begin{equation*}
  h(\lambda) \sim L_h(\lambda)  \vert \lambda \vert ^{1-2H} , \quad \lambda \to 0,
\end{equation*} where
\begin{equation*}
   L_h(\lambda) = L_R(\lambda ^{-1}) \pi^{-1} \Gamma(2H-1) \sin ( \pi -\pi H ) .
\end{equation*}
            \item If \begin{equation*}
  h(\lambda)= L_h(\lambda) \vert \lambda \vert ^{1-2H} , \quad 0< \lambda < \pi,
\end{equation*}
where $H\in (\frac 12, 1)$, and $L_h (\lambda )$ is slowly varying at the origin
in Zygmund's sense and of bounded variation on $(a, \pi)$ for any $a >0$, then
\begin{equation*}
  R(k) \sim L_R(k) \vert k \vert ^{2H-2} , \quad k\to \infty,
\end{equation*}
where
\begin{equation*}
  L_R(k) = 2L_h (k^{-1})  \Gamma(2-2H)\sin \left(\pi H - \frac 12\pi \right) .
\end{equation*}
                          \end{enumerate}
 \end{thm}




The following Breuer-Major theorem is well-known, for example, see \cite{BM} and \cite{Nourdin}:%
\begin{thm}[Breuer-Major theorem]\label{B-M}  
$ Y= (Y_t, t\in \Z) $ is a centered stationary Gaussian sequence. Set $R(k)= E(Y_0 Y_k) $. Define
\begin{equation}\label{vn def}
   V_n := \frac{1}{\sqrt n} \sum_{t=1}^n [ Y_t^2 -\E(Y_t^2 )], \quad n\ge 1.
   \end{equation}
 If $$\sigma_H^2:=2\sum_{k\in\Z} R^2(k)<\infty,$$ then
\begin{equation*}
 V_n= \frac{1}{\sqrt n} \sum_{t=1}^{n} [Y_t^2 -\E(Y_t^2 )]  \xrightarrow{law} \mathcal{N} (0,\sigma_H^2), \quad \textit{as} \, \,  n\to \infty,
\end{equation*} 
\end{thm}

The following theorem  is taken form Theorem 7.3.1 of \cite{Nourdin}, which is a corollary of Fourth Moment Berry-Ess\'een bound.

\begin{thm}\label{rate} 
Set $Z \sim \mathcal{N} (0, 1)$ and let $R(k),\,V_n$ be given in Theorem~\ref{B-M}.   Set  $v_n^2 = \E(V_n^2) $. 
Then, for all $n\ge 1$,
\begin{equation*}
  d_{TV} (V_n /v_n, Z) \le \frac{4 \sqrt 2}{v_n^2 \sqrt n} \left( \sum_{k=-n+1}^{n-1} |R(k)|^{\frac 43} \right)^{\frac 32}.
\end{equation*}
\end{thm}

The following theorem is rephrased from Theorem 5.1 and Proposition 5.2 of \cite{BNT 2010}, which gives a sufficient condition to the almost sure central limit theorem of $\bar{V}_n=\frac{V_n}{v_n}$.
\begin{thm}\label{ASCLT BN}
Let $V_n$ be given by \eqref{vn def}  and  $Z \sim \mathcal{N} (0, 1)$. Set  $v_n^2 = \E(V_n^2) $.  Assume that $R(k)\sim \abs{k}^{-\beta}L(\abs{k})$, as $\abs{k}\to \infty$, for some $\beta>\frac12$ and some slowly varying function $L(\cdot)$. 
Then $\set{\frac{V_n}{v_n}}$ satisfies an almost sure central limit theorem. 
\end{thm}

\section{Proofs of the Strong Consistency, The Asymptotic Normality  and ASCLT}
\begin{lem}\label{n-a} 
Let $(Y_t)_{t\in\mathbb N}$ and $\zeta$ be given by \eqref{solution}.
  Then for all $\varepsilon >0 $ there exists a random variable $C_{\varepsilon}$ such that
\begin{equation}\label{to-0}
     \left\vert   \zeta \sum_{t=1}^n \theta^t Y_t \right\vert
     \le C_{\varepsilon} n^{  \varepsilon} \quad \textit{a.s.}
\end{equation}  for all $n\in \N$, and moreover, $\E |C_{\varepsilon}|^p <\infty$ for all $p\ge 1$.
\end{lem}

\begin{proof} 
Fix $p\ge 1$ and denote by $\| \cdot \|_{p}$ the $L^p$-norm. Since  $\sum_{t=1}^n \theta^t Y_t$ is Gaussian,
 we have
           \begin{align*}
   \left\Vert \sum_{t=1}^n \theta^t Y_t \right\Vert_ { {2p}}
   \le c_2 \left\Vert \sum_{t=1}^n \theta^t Y_t \right\Vert _2
    =c_2 \sqrt { \sum_{i,j=1}^n \theta^{i+j} \E [ Y_i Y_j ]  } \le C_{\theta},
            \end{align*}  for all $n\in \N$.
The H\"older equality implies that
\begin{align}\label{pnorm}
     \left\Vert \zeta \sum_{t=1}^n \theta^t Y_t \right\Vert_ {p}
     \le \Vert \zeta \Vert _{{2p}}   \left\Vert \sum_{t=1}^n \theta^t Y_t \right\Vert _{{2p}}
     \le c_1  \left\Vert \sum_{t=1}^n \theta^t Y_t \right\Vert _{{2p}}  \le C_{\theta},
\end{align}
 which implies \eqref{to-0} from Lemma 2.1 of \cite{Kloeden}. In fact,  it is easy to check that the conclusion of Lemma 2.1 of \cite{Kloeden} is valid if its assumption $\alpha>0$ is changed to $\alpha\ge 0$.
\end{proof}

\noindent{\it Proof of Theorem~\ref{consist}.\,}
It is obvious that the covariance function of $Y_t$ is
\begin{equation*}
    R(k)=\mathrm{Cov}(Y_t,Y_{t+k})=\sum_{i,j=0}^\infty \theta^{i+j}\rho(k-i+j).
\end{equation*}
 Since $ \rho(k)\to 0$ as $k\to \infty$ and  when $0<\theta<1$, $$\sum_{i,j=0}^{\infty}\theta^{i+j} \rho(k-i+j)\le \sum_{i,j=0}^{\infty}\theta^{i+j} <\infty,$$
the dominated convergence theorem implies that
 \begin{equation*}
   \lim _{k\to \infty}  R(k) = 0 .
\end{equation*}
 Hence, the stationary Gaussian sequence $Y_{t}$ is ergodic. Since $\E Y_t^2 = f(\theta)$, we obtain
\begin{align*}
\lim_{n\to \infty}  \frac{1}{n}\sum_{t=1}^{n}Y_t^2   =  f(\theta)   \quad a.s.
\end{align*}
   Lemma \ref{n-a} implies that  for $0<\varepsilon <1$ there exists a random variable $C_{\varepsilon}$ such that $$\abs{  \frac 1n \zeta\sum_{t=1}^{n} \theta^t Y_t }\le C_{\varepsilon} n^{ -1+ \varepsilon},\quad a.s.$$
   Hence,  as $n\to \infty$,  $$ \frac 1n \zeta\sum_{t=1}^{n} \theta^t Y_t  \to 0$$
   almost surely.
Thus, \begin{align}\label{x2 moment limit}
\lim_{ n\to \infty}   \frac{1}{n}\sum_{t=1}^{n}X_t^2    =  f(\theta)  \quad a.s.,
\end{align}
which, together with the continuous mapping theorem, implies that  as $n\to \infty$,  $$\tilde{\theta}_n\to \theta,\quad a.s.$$
{\hfill\large{$\Box$}}
\begin{rem}\label{rem 3.2}
Since $ f(\theta),\,\theta\in (0,1) $ is a strictly increasing function, the limit \eqref{x2 moment limit} implies that when $n$ is large enough, $ \frac{1}{n}\sum_{t=1}^{n}X_t^2  $ belongs to the range of $ f(\theta)$. Hence, when $n$ is large enough,
\begin{equation}
 (0,1)\ni \tilde{\theta}_n=f^{-1}( \frac{1}{n}\sum_{t=1}^{n}X_t^2  ),\quad a.s.
\end{equation}
\end{rem}
\begin{lem}
 Under Hypothesis \ref{assume},
   the stationary solution $Y_t$ to the model (\ref{model}) exhibits linear long-range dependence. Namely,
\begin{equation}\label{long}
   h_Y(\lambda)  \sim C_{\theta,H} L(\lambda^{-1}) \vert \lambda \vert ^{ 1-2H } , \quad \textit{as} \,\, \lambda \to 0,
\end{equation} where the constant $C_{\theta , H} >0$ depends on $\theta$ and $H$.
\end{lem}
\begin{proof}
The spectral density of  the stationary solution $Y_t$ to the model (\ref{model}) satisfies
\begin{equation*}
 h_Y(\lambda) =\vert 1-\theta e^{-i\lambda} \vert ^{-2} h_{\xi}(\lambda).
 \end{equation*}
  Under Hypothesis \ref{assume}, Theorem~\ref{sp-cov}(1) implies that as $ \lambda \to 0 $,
  \begin{equation*}
h_{\xi}(\lambda) \sim C_H L(\lambda^{-1}) \vert \lambda \vert ^{ 1-2H }.
  \end{equation*}
Hence,  we obtain that  as $ \lambda \to 0 $,
$$h_Y(\lambda) \sim (1-\theta)^{-2}  h_{\xi}(\lambda) \sim C_{\theta,H} L(\lambda^{-1}) \vert \lambda \vert ^{ 1-2H }.$$
\end{proof}

\begin{lem} \label{lem 3.3}    
 Let  $Y_t$ be   the stationary solution to the model (\ref{model}) and $R(k)=\E(Y_kY_0)$. Let  $V_n $ be given in Theorem~\ref{B-M} and $v_n^2=\E(V_n^2)$.  When $ H\in (\frac 12 , \frac 34)$,
     \begin{equation}\label{fi}
\lim_{n\to \infty} v_n^2= 2\sum_{k\in \Z} R^2(k) < \infty  .
    \end{equation}
\end{lem}
\begin{proof}
It is well known the following two identities hold:
 \begin{align}
     v_n^2 &= \frac 2n \sum_{k,l=1}^n  R^2(k-l)  = 2 \sum_{|k|<n} \left(  1- \frac{|k|}{n}  \right)  R^2(k) , \nonumber\\
     \lim_{n\to \infty} v_n^2&= 2\sum_{k\in \Z} R^2(k),\quad \text{ if } \quad \sum_{k\in \Z} R^2(k) < \infty, \label{v2 expression}
  \end{align}
 see, for example, (7.2.6) of \cite{Nourdin}.

 Thus, to check (\ref{fi}) for $H\in (\frac 12 , \frac 34)$,  we need only to show the condition $\sum_{k\in \Z} R^2(k) < \infty $ holds. In fact,
Theorem~\ref{sp-cov} (2) and the identity (\ref{long}) imply that the covariance function of $Y_t$ satisfies that when $ \frac{1}{2} <H<1 $,
\begin{equation} \label{equ}
    R(k) \sim C'_{\theta,H} L(k) \vert k \vert ^ {2H-2}, \quad \textit{as} \,\, k\to \infty.
\end{equation}
Recall that $ L(k) < c \vert  k \vert ^{\delta}$ for any fixed $\delta >0 $ and $k$ large enough (see, for example, \cite[p.277]{feller71} ).
  Hence, $$ \sum_{k\in \Z} R^2(k)<\infty $$
  if and only if $$4H-4 + 2 \delta < -1.$$ Note that $\delta>0$ is arbitrary, we obtain that when $\frac12<H < \frac 34$,
the condition $\sum_{k\in \Z} R^2(k) < \infty $ holds.
\end{proof}

\noindent{\it Proof of Theorem~\ref{normal}.\,}
Using the identity (\ref{solution}), we have that
\begin{equation}\label{deompos}
 \frac{1}{\sqrt{n}}  \sum_{t=1}^{n}\left(X_t^2-f(\theta)\right) = \frac{1}{\sqrt{n}}  \sum_{t=1}^{n}\left(Y_t^2-f(\theta)\right)  + \frac{2\zeta}{\sqrt n}\sum_{t=1}^{n} \theta^t Y_t +  \frac{\zeta^2}{\sqrt n}\sum_{t=1}^{n} \theta^{2t}.
\end{equation}
Theorem \ref{B-M} and Lemma~\ref{lem 3.3} imply that  when $ \frac{1}{2} <H< \frac{3}{4} $, as $n\to \infty$,
  \begin{equation*}
  \frac{1}{\sqrt{n}}  \sum_{t=1}^{n}\left(Y_t^2-f(\theta)\right)  \xrightarrow{law} \mathcal{N} ( 0, \sigma_H^2).
  \end{equation*}
 Lemma \ref{n-a} implies that  as $n\to \infty$,
  \begin{equation}\label{limit 1}
   \abs{ \frac{\zeta}{\sqrt n}\sum_{t=1}^{n} \theta^t Y_t }\le C_{\epsilon} n^{\epsilon -\frac12}\to 0.  \end{equation}
It is clear that  as $n\to \infty$,
 \begin{equation} \label{limit 11}
 \frac{\zeta^2}{\sqrt n}\sum_{t=1}^{n} \theta^{2t}\to 0.  \end{equation}
Substituting the above three limits into (\ref{deompos}), we deduce from Slutsky's theorem that as $n\to \infty$,
\begin{equation*}
 \frac{1}{\sqrt{n}}  \sum_{t=1}^{n}\left(X_t^2-f(\theta)\right) \xrightarrow{law} \mathcal{N} ( 0, \sigma_H^2 ),
\end{equation*} which implies the desired (\ref{1}) from the delta method.
 {\hfill\large{$\Box$}}

\noindent{\it Proof of Theorem~\ref{ASCLT 01}.\,} The ASCLT can be obtained by arguments similar to those of Theorem 4.3 in \cite{ces 15}.

First, Eq.\eqref{equ} and Theorem~\ref{ASCLT BN} imply that when $ \frac{1}{2} <H<\frac34 $,
\begin{equation*}
\set{\frac{V_n}{v_n}\,:n\ge 1}
\end{equation*}
satisfies the ASCLT.

Second, by \eqref{v2 expression}, \eqref{deompos}-\eqref{limit 11}, we have that $$\frac{1}{\sqrt{n}\sigma_H}  \sum_{t=1}^{n}\left(X_t^2-f(\theta)\right) =\frac{V_n}{v_n}\frac{v_n}{\sigma_H}+ \frac{2\zeta}{\sqrt n \sigma_H}\sum_{t=1}^{n} \theta^t Y_t +  \frac{\zeta^2}{\sqrt n \sigma_H}\sum_{t=1}^{n} \theta^{2t},$$
which, together with Theorems~3.1 and 3.2 of \cite{ces 15}, implies that \begin{equation*}
\set{\frac{1}{\sqrt{n}\sigma_H}  \sum_{t=1}^{n}\left(X_t^2-f(\theta)\right) \,:n\ge 1}
\end{equation*}
satisfies the ASCLT.

Third,  the mean value theorem implies that
$$ \frac{ f'(\theta) \sqrt n  (\tilde{\theta}_n  -\theta)}{ \sigma_H }=\frac{ f'(\theta)  \sqrt{n}}{f'(\eta_n) \sigma_H }\Big(\frac{1}{n}\sum_{t=1}^n X_t^2- f(\theta) \Big)
=\frac{ f'(\theta) }{f'(\eta_n)}\frac{1}{\sqrt{n}\sigma_H}  \sum_{t=1}^{n}\left(X_t^2-f(\theta)\right) $$
where $\eta_n$ is a random variable between $\frac{1}{n}\sum_{t=1}^n X_t^2$ and $f(\theta)$. The convergence \eqref{x2 moment limit} leads to
$\frac{ f'(\theta) }{f'(\eta_n)}\to 1$ almost surely as $n\to \infty$. It follows from Theorems~3.1 of \cite{ces 15} that
 \begin{equation*}
\set{\frac{ f'(\theta) \sqrt n  (\tilde{\theta}_n  -\theta)}{ \sigma_H } \,:n\ge 1}
\end{equation*}
satisfies the ASCLT.

 {\hfill\large{$\Box$}}

\section{The Berry-Ess\'een Bound}
The following Fourth Moment Berry-Ess\'een bound is similar to Corollary 7.4.3 of \cite{Nourdin}.
\begin{prop}\label{prop Fourth Moment B-E}
Let  $V_n $ be given in Theorem~\ref{B-M} and $v_n^2=\E(V_n^2)$.
When $H\in (\frac 12,  \frac 34  ) $,  there exists a constant $c_H > 0$ such
that, for all $n \ge 2$:
\begin{equation}\label{upper bound be}
   d_{TV} (V_n /v_n, Z) \le c_H  \varphi_1 (n) ,\end{equation}
   where
   \begin{equation*}
  \varphi_1 (n)=
   \left\{
         \begin{array}{ll}
n^ {- \frac 12}   ,   & \quad \textit{if} \,\,  H\in   \left( \frac 12, \frac 58 \right) ;\\
 \frac{1}{n^{3-4H-}}, & \quad \textit{if} \,\,  H\in \left[ \frac 58, \frac 34 \right) .
            \end{array}
   \right.
   \end{equation*}
\end{prop}
\begin{rem}
When the  function slowly varying at infinity $L(k)$ degenerates to a positive constant as $k$ large enough, we can improve the upper bound  $\frac{1}{n^{3-4H-}}$ to $ \frac{1}{n^{3-4H}}$ in the case of $H\in \left( \frac 58, \frac 34 \right)$.
\end{rem}
\begin{proof}
Since $ L(k) < c \vert  k \vert ^{\delta}$ for any fixed $\delta >0 $ and $k$ large enough, we have that
\begin{equation}\label{upper bound Rk}
\abs{R(k) }^{\frac43}\le  C |k|^{ \frac43 (2H-2 + \delta)}. \end{equation}
When  $  H\in (\frac12,\,\frac58) $, we can take $\delta >0 $ small enough such that $ \frac43 (2H-2 + \delta)<-1 $ and the following series converges,
$$\sum_{k\in \Z}\abs{R(k) }^{\frac43} <\infty .$$
Together with the limit \eqref{fi}, we obtain form Theorem \ref{rate} that the desired bound \eqref{upper bound be} holds.\\
When $ 5/8 \le H < 3/4$, the inequality \eqref{upper bound Rk} implies that
$$\left( \sum_{k=-n+1}^{n-1} |R(k)|^{\frac 43} \right)^{\frac 32} \le C n^{4H-\frac 52+2\delta}. $$
Again by \eqref{fi} and Theorem \ref{rate}, we have the desired bound \eqref{upper bound be} since $\delta>0$ is arbitary.
\end{proof}
The following result is Lemma~2 of \cite{changrao89}.
\begin{lem}\label{lem 4.1}
For any random variable $\xi$, $\eta$ and real constant $a>0$,
\begin{equation}
\sup_{u\in \R}\abs{P(\xi+\eta\le u)-\Phi(u)}\le \sup_{u\in \R}\abs{P(\xi \le u)-\Phi(u)}+ P(\abs{\eta}>a)+\frac{a}{\sqrt{2\pi}},
\end{equation} where $\Phi(u)$ stands for the standard normal distribution function.
\end{lem}
\noindent{\it Proof of Theorem~\ref{B-E}.\,}
The Berry-Esseen bound (\ref{l34}) can be obtained by arguments similar to those of Theorem 3.2 in \cite{Sott}. 
Denote  \begin{equation*}
  A:= P\left\{ \frac{ f'(\theta) \sqrt n  (\tilde{\theta}_n  -\theta)}{ \sigma_H } \le z \right\} -   P\{ Z\le z \}.
\end{equation*}
Since $\tilde{\theta}_n \in( 0,\,1) $ almost surely (see Remark~\ref{rem 3.2}), we shall suppose that  $z\in \mathcal{D}$:
\begin{equation}\label{Domain D}
\mathcal{D}:=\set{z:\,\,-\frac { f'(\theta)\sqrt n}{\sigma_H}  \theta  < {z} < \frac { f'(\theta)\sqrt n}{\sigma_H} \Big(\theta \wedge (1-\theta)\Big)}.
\end{equation}
Otherwise, the upper-tail inequality for standard normal distribution  $$P(Z \ge t) \le \frac{e^{-\frac{t^2}{2}}} {t\sqrt{2\pi}},\quad t>0$$ yields
\begin{equation*}
   |A|= P\{ Z> \abs{z} \} \le \frac {C}{\sqrt n}.
\end{equation*}
Since $f(\theta)$ is strictly increasing and continuous,  we have by \eqref{estimator} for the formula of $\tilde{\theta}_n$
\begin{align*}
A & = P\left\{ \frac{ f'(\theta) \sqrt n  (\tilde{\theta}_n  -\theta)}{ \sigma_H } \le z \right\} -   P\{ Z\le z \}  \\
  & =   P\left\{ \tilde{\theta}_n    \le   \theta +  \frac{ \sigma_H}{f'(\theta) \sqrt n } z    \right\} -   P\{ Z\le z \}  \\
 & =   P\left\{ \frac 1n  \sum_{t=1}^n  X_t^2 \le  f \left(  \theta +  \frac{ \sigma_H}{f'(\theta) \sqrt n } z  \right)  \right\}  -   P\{ Z\le z \}  \\
& =   P\left\{ \frac 1n  \sum_{t=1}^n  X_t^2  - f(\theta)
\le  f \left(  \theta +  \frac{ \sigma_H}{f'(\theta) \sqrt n } z  \right)  - f(\theta)  \right\}  -   P\{ Z\le z \}   .
\end{align*}
We take the short-hand notation
\begin{equation}\label{uz expression}
  u(z) = \frac{ \sqrt n } {\sigma_H}  \left[ f \left(  \theta +  \frac{ \sigma_H}{f'(\theta) \sqrt n } z  \right)  - f(\theta) \right]
   \end{equation}  and
 \begin{equation*}
  w = \frac{2\zeta}{\sigma_H \sqrt n}\sum_{t=1}^{n} \theta^t Y_t + \frac{\zeta^2}{\sigma_H \sqrt n}\sum_{t=1}^{n} \theta^{2t}.
 \end{equation*}
By \eqref{solution} the relationship between $X_t$ and $Y_t$ and \eqref{vn def} the formula of $V_n$,  we have
    \begin{align*}
  |A| &= \left| P\left\{ \frac {V_n} {\sigma_H} + w \le u \right\} - P\{ Z\le z \} \right|  \\
      & \le \left| P\left\{ \frac {V_n} {\sigma_H} + w \le u \right\} - \Phi(u) \right| + \left| \Phi(u) - P\{ Z\le z \} \right|.
          \end{align*} 
        The second term is bounded by $ C n^{-\frac 12 }$ from Lemma \ref{A3} below. Hence, we need only show that
\begin{equation}  \label{zhongjian bound}
 \sup_{u\in \R} \left| P\left\{ \frac {V_n} {\sigma_H} + w \le u \right\} - \Phi(u) \right| \le C  \varphi (n).
 \end{equation}
In fact,  Lemma~\ref{lem 4.1} implies that  for any $\gamma >0$,
         \begin{align}\label{b-e bound}
          \sup_{u\in \R} \left| P\left\{ \frac {V_n} {\sigma_H} + w \le u \right\} - \Phi(u) \right|
  & \le  \sup_{u\in \R} \left| P\left\{ \frac {V_n} {\sigma_H}  \le u \right\} - \Phi(u) \right|
  +   P\left\{ |w|  > n^{- \gamma} \right\}    +   \frac{ n^{- \gamma} } { \sqrt{2 \pi} } .
         \end{align}
    By Proposition~\ref{prop Fourth Moment B-E}, the Fourth moment Berry-Ess\'{e}en bound, we have that
    $$ \sup_{u\in \R} \left| P\left\{ \frac {V_n} {\sigma_H}  \le u \right\} - \Phi(u) \right| \le C  \varphi_1 (n) .$$
Chebyshev's inequality implies that
        \begin{align*}
 P\left\{ |w|  > n^{-\gamma} \right\}
  \le n^{p \gamma}  \E|w|^p. 
        \end{align*}
The triangular inequality and the inequality \eqref{pnorm}  imply that
\begin{align*}
\norm{w}_p\le \frac{2}{\sigma_H \sqrt n}\norm{\zeta\sum_{t=1}^{n} \theta^t Y_t}_{p} + \frac{1}{\sigma_H \sqrt n}\norm{\zeta^2\sum_{t=1}^{n} \theta^{2t}}_p\le \frac{C}{\sqrt n},
\end{align*} from which we have that
\begin{align*}
P\left\{ |w|  > n^{-\gamma} \right\}\le C n^{ p(\gamma-\frac12)}.
\end{align*} We take $\gamma < \frac 12$, $p$ large enough, and substitute the above inequalities into \eqref{b-e bound} to get \eqref{zhongjian bound}.
{\hfill\large{$\Box$}}

\begin{lem}\label{A3}
Let $u(z) $ be given by \eqref{uz expression} and  the interval $\mathcal{D}$ given by \eqref{Domain D}.
Then there exists some positive number $C_{\theta, H}$ independent of $n$ such that
\begin{equation*}
\sup_{z\in \mathcal{D}} | \Phi(u(z)) - \Phi(z) |  \le  \frac {C_{\theta, H}} {\sqrt n}.
\end{equation*}
\end{lem}

\begin{proof} We follow the line of the proof of Theorem 3.2 in  \cite{Sott}. %
 By the mean value theorem, there exists some number $\eta \in ( \theta , \theta +  \frac{ \sigma_H}{ f'(\theta) \sqrt n }z )$ when $z>0$ or $\eta \in ( \theta +  \frac{ \sigma_H}{ f'(\theta) \sqrt n }z,\,\theta  )$ when $z<0$ such that
    \begin{align*}
  u(z) = \frac{ \sqrt n } {\sigma_H}  \left[ f \left(  \theta +  \frac{ \sigma_H}{f'(\theta) \sqrt n } z  \right)  - f(\theta) \right]
 = \frac{f'(\eta)} {f'(\theta)} z.
   \end{align*} Since $f(\theta)$ is convex, we have  for any $z$,  $ u(z)\ge z $. Hence,
\begin{equation*}
   | \Phi(u)- \Phi(z) | 
   = \frac{1}{\sqrt{2\pi}}\int_{z}^{\frac{f'(\eta)} {f'(\theta)} z} e^{-\frac {t^2} {2}} \,\mathrm d t  .
\end{equation*}
Since the function $$ f(x,z) = z^2 e^{-\frac{x^2 z^2}{2}} \vert x-1 \vert $$ is uniformly bounded, we have when $z<0$,
\begin{align*}
 \frac{1}{\sqrt{2\pi}}\int_{z}^{\frac{f'(\eta)} {f'(\theta)} z} e^{-\frac {t^2} {2}} \, \mathrm d t
     \le \frac{1}{\sqrt{2\pi}} \abs{ z}  e^{-  \frac{ z^2}{2} \big( \frac{f'(\eta)} { f'(\theta) } \big)^2}\abs{\frac{f'(\eta)}{f'(\theta)}-1}
\le \frac{C_{\theta, H} }{|z|} .
\end{align*} Thus, 
we obtain that
\begin{align*}
 \sup_{- \frac { f'(\theta)\sqrt n}{\sigma_H} \theta <z\le - \frac { f'(\theta)\sqrt n}{2\sigma_H} \theta} | \Phi(u) - \Phi(z) |  \le  \frac {C_{\theta, H} } {\sqrt n}.
\end{align*}
When $  - \frac { f'(\theta)\sqrt n}{2\sigma_H} \theta <z <\frac { f'(\theta)\sqrt n}{2\sigma_H} \Big(\theta \wedge (1-\theta)\Big)$,
using the mean value theorem and making the change of variable $t = z^2 s$,
together with the fact that $f_2(s, z) = z^2 e^{-\frac{s^2 z^4}{2}}$ is also uniformly bounded,
we conclude that there exists a number $$\eta _1 \in (\theta , \eta)\subset [\theta,\, \theta+\frac12  \Big(\theta \wedge (1-\theta)\Big)] $$ such that
\begin{align*}
  \int_{z}^{\frac{f'(\eta)} {f'(\theta)} z} e^{-\frac {t^2} {2}} \, \mathrm d t &
= \int_{\frac{1}{z}}^{\frac{f'(\eta)} {f'(\theta)} \frac{1}{z}} z^2 e^{-\frac {s^2 z^4} {2}} \, \mathrm d s\\
&\le C_{\theta, H}  \frac{1}{\vert z \vert} {|f'(\eta)-f'(\theta)|}   \\
&\le C_{\theta, H}   \frac{1}{\vert z \vert} \abs{f''(\eta _1)} \frac{  \vert z \vert}{  \sqrt n}
\\ &\le \frac{C_{\theta, H} }{ \sqrt n} ,
\end{align*} since the second derivative function $f''$  is bound on the close interval $ [\theta,\, \theta+\frac12  \Big(\theta \wedge (1-\theta)\Big)] $ .
When $ z\ge \frac { f'(\theta)\sqrt n}{2\sigma_H} \Big(\theta \wedge (1-\theta)\Big)$,
   \begin{equation*}
\frac{1}{\sqrt{2\pi}}   \int_{z}^{\frac{f'(\eta)} {f'(\theta)} z}  e^{-\frac {t^2} {2}} \,   \mathrm d t
\le  \frac{1}{\sqrt{2\pi}}   \int_{z}^{\infty}   e^{-\frac {t^2} {2}} \,   \mathrm d t \le  \frac{C_{\theta, H} }{\sqrt n}.
       \end{equation*}
\end{proof}

\vskip 0.2cm {\small {\bf  Acknowledgements}:
This research is partly supported by NSFC (No.11961033).
}


\end{document}